\documentclass{amsart}
\usepackage{amsmath, amsthm, amssymb}
\usepackage{verbatim}
\usepackage{hyperref}

\usepackage{enumerate}
\usepackage[T1]{fontenc}
\usepackage[utf8]{inputenc}
\usepackage[english]{babel}

\usepackage{tikz}
\usetikzlibrary{calc}

\usepackage{kantlipsum}
\allowdisplaybreaks

\theoremstyle{plain}
    \newtheorem{theorem}{Theorem}[section]

    \newtheorem*{introtheorem}{Theorem}
    \newtheorem*{introconjecture}{Conjecture}
    \newtheorem{lemma}[theorem]{Lemma}

\theoremstyle{definition}

\theoremstyle{remark}

\begin{document}
\title{Linear restrictions on cone polynomials}
\author{Weibo Fu}
\address{W. Fu: Department of Mathematics, The University of Science and Technology of China, Hefei, Anhui, China}
\email{fuweibo@mail.ustc.edu.cn}
\author{Zipei Nie}
\address{Z. Nie: Department of Mathematics, Princeton University, Princeton NJ}
\email{znie@math.princeton.edu}
\date{\today}
\begin{abstract}
For a set $S$ of $d$ points in the $n$-dimensional projective space over a field of characteristic zero, we prove that the polynomials of degree $d$ whose zero sets are cones over $S$ do not span the vector space of polynomials of degree $d$ vanishing on $S$, if $d$ is odd and $d\ge 3$. Furthermore, they span a subspace of codimension at least two, if $n=2$, $d=1\pmod 4$ and $d\ge 5$.
\end{abstract}
\maketitle
\section{Introduction}
The sets of finite generic points in the projective space have been studied extensively and serve as fundamental examples in the multivariate polynomial interpolation theory, where people are particularly interested in the vanishing ideals of those sets. See, e.g., \cite{ciliberto}.

In our setting, let $S=\{P_i:1\le i\le d\}$ be a set of $d$ generic points in the projective space $\mathbf{P}^n$ over a field $\mathbb{F}$ of characteristic zero. To construct a homogeneous polynomial of degree $d$ vanishing on $S$, we may first choose an $(n-2)$-plane $\pi$ in $\mathbf{P}^n$, and then take the product of $d$ linear functions $L_i$, such that $L_i$ vanishes at $P_i$ and $\pi$, where $i$ is an integer with $1\le i \le d$. The polynomials obtained by the proceeding procedure are called \textbf{cone polynomials}. We concern whether the cone polynomials span the vector space of the homogeneous polynomials of degree $d$ vanishing on $S$.

It is a standard fact that the vector space of the homogeneous polynomials of degree $d$ vanishing on $S$ has the expected dimension $\binom{d+n}{n}-d$.  Let $V(n,d)$ denote the span of the cone polynomials, then our problem is equivalent to ask whether we have $$\dim V(n,d)= \binom{d+n}{n}-d.$$

In the case where $n=2$, motivated by numerical experiments which were firstly done by Bo Ilic, it was conjectured that the answer is affirmative if and only if $d$ is even or $d=1$ (personal communication with Robert Lazarsfeld, December 2014).

\begin{introconjecture}
For each integer $d$ with $d\ge 2$, we have $$\dim V(2,d)=\begin{cases} \binom{d+2}{2}-d &\mbox{, if }d=0\pmod 2\\\binom{d+2}{2}-d-2 &\mbox{, if }d=1\pmod 4\\\binom{d+2}{2}-d-1 &\mbox{, if }d=3\pmod 4. \end{cases}$$
\end{introconjecture}

This conjecture states that there is exactly one hidden linear restriction on the cone polynomials if $n=2$ and $d=3\pmod 4$, and there are exactly two if $n=2$, $d=1\pmod 4$ and $d\ge 5$. We prove the existence of these linear restrictions in subsequent sections.

In Section \ref{first-part}, we show the following theorem holds, which means the answer to our original problem is negative when $d$ is odd and $d\ge 3$.

\begin{introtheorem}
For each odd integer $d$ with $d\ge 3$, we have $$\dim V(n,d)\le \binom{d+n}{n}-d-1.$$
\end{introtheorem}

In Section \ref{second-part}, we prove the existence of an extra linear restriction when $n=2$, $d=1\pmod 4$ and $d\ge 5$. In other words, we prove the following theorem.

\begin{introtheorem}
For each integer $d$ with $d=1\pmod 4$ and $d\ge 5$, we have $$\dim V(2,d)\le\binom{d+2}{2}-d-2.$$
\end{introtheorem}

\section{The first linear restriction}\label{first-part}
For each integer $i$ with $1\le i \le d$, let $\tilde{P}_i$ in $\mathbb{F}^{n+1}\backslash \{0\}$ represent $P_i$ in homogeneous coordinates, and let $\tilde{P}_i\cdot\nabla$ denote the directional derivative operator along $\tilde{P}_i$.

Suppose that $\{Q_j\in \mathbf{P}^n: 1\le j\le n-1\}$ is an affine basis of the $(n-2)$-plane $\pi$. For each integer $j$ with $1\le j \le n-1$, let $\tilde{Q}_j$ in $\mathbb{F}^{n+1}\backslash \{0\}$ represent $Q_j$ in homogeneous coordinates. For each integer $i$ with $1\le i\le d$, define the linear function $l_{i}$ in $\mathbb{F}[x_1,\ldots,x_{n+1}]$ by $l_i (v)=\det[v,\tilde{P}_i,\tilde{Q}_1,\ldots, \tilde{Q}_{n-1}]$, then $l_{i}$ vanishes at $P_i$ and $\pi$. In addition, $l_i$ is nonzero if and only if $P_i\not \in \pi$.

Since $(\tilde{P}_i\cdot\nabla) l_{i'}=-(\tilde{P}_{i'}\cdot\nabla) l_i=\det[\tilde{P}_i,\tilde{P}_{i'},\tilde{Q}_1,\ldots, \tilde{Q}_{n-1}]$ for each pair of integers $(i,i')$ with $1\le i,i'\le d$. If $d$ is odd, then
\begin{align*}
&\left(\prod_{i=1}^d \tilde{P}_i\cdot\nabla\right)\left(\prod_{i=1}^d l_i\right)\\
=&\sum_{\sigma \in \mathcal{S}_d}\prod_{i=1}^d (\tilde{P}_i\cdot\nabla) l_{\sigma(i)}\\
=&\frac{1}{2}\sum_{\sigma \in \mathcal{S}_d}\left(\prod_{i=1}^d(\tilde{P}_i\cdot\nabla) l_{\sigma(i)}-\prod_{i=1}^d(\tilde{P}_{\sigma(i)}\cdot\nabla) l_i \right)\\
=&\frac{1}{2}\sum_{\sigma \in \mathcal{S}_d}\prod_{i=1}^d (\tilde{P}_i\cdot\nabla) l_{\sigma(i)}-\frac{1}{2}\sum_{\sigma^{-1}\in \mathcal{S}_d}\prod_{i=1}^d (\tilde{P}_i\cdot\nabla) l_{\sigma(i)}\\
=&0,
\end{align*}
where, as usual, $\mathcal{S}_d$ stands for the symmetric group of degree $d$.

Following the construction of cone polynomials, for each integer $i$ with $1\le i\le d$, let $L_i$ be any linear function in $\mathbb{F}[x_1,\ldots,x_{n+1}]$ vanishing at $P_i$ and $\pi$. If $\pi\cap S=\emptyset$, then $L_i$ is a scalar multiple of $l_i$ for each integer $i$ with $1\le i\le d$. Hence if $\pi\cap S=\emptyset$ and $d$ is odd, then
$$\left(\prod_{i=1}^d \tilde{P}_i\cdot\nabla\right)\left(\prod_{i=1}^d L_i\right)=0.$$
Otherwise, we assume $P_1\in \pi$ without loss of generality, then $(\tilde{P}_1\cdot\nabla) L_i=0$ for each $i$ with $1\le i\le d$, thereby we still have $$\left(\prod_{i=1}^d \tilde{P}_i\cdot\nabla\right)\left(\prod_{i=1}^d L_i\right)=\sum_{\sigma \in \mathcal{S}_d}\prod_{i=1}^d (\tilde{P}_i\cdot\nabla) L_{\sigma(i)}=0.$$ Therefore the following lemma holds.

\begin{lemma}\label{relation}
For each odd positive integer $d$ and each cone polynomial $f$, we have
$$\left(\prod_{i=1}^d \tilde{P}_i\cdot\nabla\right) f=0.$$
\end{lemma}

Our next lemma shows that the linear relation in Lemma $\ref{relation}$ indeed gives us a nontrivial linear restriction on the cone polynomials if $d$ is odd and $d\ge 3$.
\begin{lemma}\label{independence} For each integer $d$ with $d\ge 3$, there exists a polynomial $f_0\in \mathbb{F}[x_1,\ldots, x_{n+1}]$ of degree $d$, vanishing on $S$, such that $$\left(\prod_{i=1}^d \tilde{P}_i\cdot\nabla\right)f_0 \neq 0.$$
\end{lemma}
\begin{proof}
Define the polynomial $f_0$ in $\mathbb{F}[x_1,\ldots, x_{n+1}]$ by $$f_0(v)=\prod_{i=1}^{d}\det[\phi(v), \phi(\tilde{P}_i) ,\phi(\tilde{P}_{i+1})],$$ where the indices are taken modulo $d$ and the projection $\phi:\mathbb{F}^{n+1}\rightarrow \mathbb{F}^3$ is defined by $\phi(x_1,\ldots,x_{n+1})=(x_1,x_2,x_3)$, then $f_0$ vanishes on $S$.

To derive the result $$\left(\prod_{i=1}^d \tilde{P}_i\cdot\nabla\right)f_0\neq 0$$ for generic $a_i$, $b_i$ and $c_i$, it suffices to show the left hand side is nonzero in case $\tilde{P}_i=(1,i,i^2,\ldots,i^n)\in\mathbb{Q}^{n+1}\subseteq\mathbb{F}^{n+1}$ for each integer $i$ with $1\le i\le d$. In this case, by Vandermonde determinant identity, we have
\begin{align*}
&\left(\prod_{i=1}^d \tilde{P}_i\cdot\nabla\right)f_0
\\=&\sum_{\sigma \in \mathcal{S}_d}\prod_{i=1}^{d}(\tilde{P}_{\sigma(d)}\cdot\nabla)\det[\phi(v), \phi(\tilde{P}_i) ,\phi(\tilde{P}_{i+1})]
\\=&\sum_{\sigma \in \mathcal{S}_d}\prod_{i=1}^{d}\det[\phi(\tilde{P}_{\sigma(d)}), \phi(\tilde{P}_i) ,\phi(\tilde{P}_{i+1})]
\\=&\sum_{\sigma \in \mathcal{S}_d}(d-1)(d-\sigma(d))(\sigma(d)-1)\prod_{i=1}^{d-1} (\sigma(i)-i)(\sigma(i)-i-1)
\\>&0
\end{align*}
as rational numbers. The strictness of the last inequality comes from the fact that since $d\ge 3$ we get a positive term by taking $\sigma(i)=i-1$ modulo $d$.
\end{proof}
By Lemma \ref{relation} and Lemma \ref{independence}, we obtain the following theorem.
\begin{theorem}\label{odd}
For each odd integer $d\ge 3$, we have $$\dim V(n,d)\le \binom{d+n}{n}-d-1.$$
\end{theorem}
\section{The second linear restriction}\label{second-part}
In this section, we assume $n=2$. Let $(y_1,y_2,y_3)$ denote the coordinates of $\tilde{Q}_1$. We view $\prod_{i=1}^d l_i$ as a polynomial in $\mathbb{F}[x_1,x_2,x_3][y_1,y_2,y_3]$, then it is homogeneous of degree $d$. Additionally, the coefficients of $\prod_{i=1}^d l_i$ are homogeneous polynomials of degree $d$ in $\mathbb{F}[x_1,x_2,x_3]$.

Since the evaluation of $\prod_{i=1}^d l_i$ at each point $(y_1,y_2,y_3)$ in $\mathbb{F}^3$ gives a cone polynomial, by applying difference operations, for each triple $(\alpha_1, \alpha_2, \alpha_3)$ of nonnegative integers with $\alpha_1+\alpha_2+\alpha_3=d$, the coefficient of the $y_1^{\alpha_1}y_2^{\alpha_2}y_3^{\alpha_3}$ term in $\prod_{i=1}^d l_i$ is the span of cone polynomials.

Conversely, we prove every cone polynomial is in the span of the coefficients of $\prod_{i=1}^d l_i$. As in Section \ref{first-part}, if $Q_1\not\in S$ then a cone polynomial $\prod_{i=1}^d L_i$ is a scalar multiple of the evaluation of $\prod_{i=1}^d l_i$ at $(y_1,y_2,y_3)$, and thereby it is in the span of the coefficients of $\prod_{i=1}^d l_i$. Otherwise, a cone polynomial is the evaluation of a directional derivative of $\prod_{i=1}^d l_i$ at $(y_1,y_2,y_3)$, hence it is also in the span of the coefficients of $\prod_{i=1}^d l_i$.

Consider a matrix $M(d)$ whose rows and columns are indexed by the tuples $(\alpha_1, \alpha_2, \alpha_3)$ of nonnegative integers with $\alpha_1+\alpha_2+\alpha_3=d$. Let the entry in the row $(\alpha_1, \alpha_2, \alpha_3)$ and the column $(\beta_1,\beta_2,\beta_3)$ of $M(d)$ be the coefficient of the $x_1^{\alpha_1} x_2^{\alpha_2} x_3^{\alpha_3}\cdot y_1^{\beta_1} y_2^{\beta_2} y_3^{\beta_3}$ term in $\prod_{i=1}^d l_i$ as a polynomial in $\mathbb{F}[x_1,x_2,x_3,y_1,y_2,y_3]$. Then the dimension of $V(2,d)$ equals the rank of $M(d)$, and $M(d)$ is skew-symmetric if $d$ is odd. Because the rank of a skew-symmetric matrix is even, the dimension of $V(2,d)$ is even if $d$ is odd.

Therefore, by Theorem \ref{odd}, we obtain following theorem.

\begin{theorem}\label{1mod4}
For each integer $d$ with $d=1\pmod 4$ and $d\ge 5$, we have $$\dim V(2,d)\le\binom{d+2}{2}-d-2.$$
\end{theorem}

\section*{acknowledgments}
We thank Professor Robert Lazarsfeld for introducing us this interesting problem and giving us much advice.

\end{document}